\documentclass[a4paper, 11pt]{amsart}
\usepackage{amssymb,amsmath,amsthm,lmodern,cancel, calrsfs}

\usepackage{graphicx}
\usepackage[english]{babel}
\usepackage[T1]{fontenc}
\usepackage{textcomp}
\usepackage[dvipsnames]{xcolor}
\usepackage[shortlabels]{enumitem}
\usepackage{thmtools, mathtools} 
\usepackage[colorlinks=true,hyperindex=true]{hyperref}
\usepackage{setspace}

\author{Yilin Wang}
\address{Department of Mathematics\\
 ETH\\ 
 R\"amistrasse 101,
Z\"urich 8092, Switzerland}
\email[]{yilin.wang@math.ethz.ch}
\title[Loewner energy, conformal restriction and Werner's measure]{A note on Loewner energy, conformal restriction and Werner's measure on self-avoiding loops}
\thanks{Supported by the Swiss National Science Foundation grant \#175505.}
\date{}

\newtheorem{thm}{Theorem}[section]
\newtheorem{cor}[thm]{Corollary}
\newtheorem{lem}[thm]{Lemma}
\newtheorem{prop}[thm]{Proposition}

\newtheorem{theorem}{Theorem}

\setlist[enumerate]{topsep = 1ex, leftmargin=1.5cm}

\theoremstyle{definition} 
\theoremstyle{remark} 
\theoremstyle{remark} \newtheorem*{rk}{Remark}
\theoremstyle{remark} 
\theoremstyle{remark} 

\declaretheoremstyle[
    spaceabove=-2pt, 
    spacebelow=6pt, 
    headfont=\normalfont\bfseries, 
    bodyfont = \normalfont,
    postheadspace=1em, 
    qed=$\square$, 
    headpunct={\,:}]{mypf}

 {\endlist}

\newcommand{\abs}[1]{\left\lvert #1 \right \rvert}

\newcommand{\mc}[1]{\mathcal{#1}}
\newcommand{\m}[1]{\mathbb{#1}}

\newcommand{\rar}[0]{\rightarrow}

\def\g{\gamma}
\def\G{\Gamma}
\def\d{\delta}
\def\D{\Delta}

\def\k{\kappa}

\def\s{\sigma}

\def\vare{\varepsilon}

\def\Chat{\hat{\m{C}}}

\def\SLE{\operatorname{SLE}}

\def\dd{\,\mathrm{d}}
\def\vol{\mathrm{vol}}

\def\detz{\mathrm{det}_{\zeta}}

\begin{document}

\begin{abstract} 
In this note, we establish an expression of the Loewner energy of a Jordan curve on the Riemann sphere in terms of Werner's measure on simple loops of $\SLE_{8/3}$ type. 
 The proof is based on a formula for the change of the Loewner energy under a conformal map that is reminiscent of the restriction properties derived for SLE processes. 
\end{abstract}
   
\maketitle

\section{Introduction}
Loewner's idea \cite{Loewner1923} of encoding a simple curve into a real-valued driving function provides a powerful tool in the analysis of univalent functions on the unit disk $\m D$ by considering the Loewner flow associated to the boundary of their image.  
This idea led to the solution of the Bieberbach conjecture by De Branges \cite{DeBranges1985}, and 
has also more recently received a lot of attention since 1999 with the construction of random fractal simple curves, the SLEs, by Oded Schramm \cite{schramm2000scaling}. 

The study of the Loewner energy lies at the interface between classical Loewner theory and the SLE theory, as it is an deterministic quantity associated to regular deterministic curves 
but simultaneously reflects the large deviation structure of SLEs \cite{W1}, driven by a vanishing multiple of Brownian motion.
The Loewner energy is studied in \cite{Dubedat2007commutation,FrizShekhar2017,W1,RW,W2} in various similar settings. 
Let us review briefly its definition for chords, and then for loops as introduced by Rohde and the author \cite{RW}.

   Let $\g$ be a simple curve (chord) from $0$ to $\infty$ in the upper half-plane $\m H$, one chooses to parametrize $\g$ by the half-plane capacity of $\g [0,t]$ seen from infinity,   which means the conformal map $g_t$ from $\m H \backslash \g[0,t]$ to $\m H$, that is normalized near infinity by $g_t (z) = z + o(1)$ does satisfy $g_t (z) = z + 2t/z + o(1/z)$.
   The function $g_t$ can be extended continuously to the tip $\g_t$ of the slit $\g[0,t]$ which enables to define 
   $W(t) := g_t (\g_t)$. 
   The function $W: \m R_+ \to \m R$ is called the driving function of $\g$. 
The \emph{Loewner energy} in  $(\m H, 0 ,\infty)$ of the chord $\gamma$ is defined to be 
$$I_{\m H, 0 ,\infty}(\gamma) := I(W) : = \frac{1}{2}\int_{0}^{\infty} W'(t)^2 dt$$
when $W$ is absolutely continuous, and is $\infty$ otherwise. 
Notice that $I(W)$ is the action functional of the standard Brownian motion $B$, therefore the large deviation rate function of the law of $\sqrt \k B$ under appropriate norm as $\k \to 0$.

The definition extends to a chord $\g$ connecting two prime ends $a$ and $b$ in a simply connected domain $D\subset \m C$, via a uniformizing conformal map $\varphi: D \to \m H$ such that $\varphi(a) = 0$ and $\varphi (b) = \infty$, that is
$$I_{D,a,b} (\g) : = I_{\m H, 0, \infty} (\varphi (\g)).$$
Following \cite{RW}, we define \emph{Loewner loop energy} via a limiting procedure. 
Let $\G: [0,1] \to \m C$ be an oriented Jordan curve with a marked point $\G (0) = \G(1) \in \G$. 
For every $\vare>0$, $\G [\vare, 1]$ is a chord connecting $\G(\vare)$ to $\G (1)$ in the simply connected domain $\Chat \backslash \G[0, \vare]$, where $\Chat = \m C \cup \{\infty\} \simeq S^2$ is the Riemann sphere.
The rooted Loewner loop energy is then defined as
$$I^L(\G, \G(0)): = \lim_{\vare \to 0} I_{\Chat \backslash \G[0, \vare], \G(\vare), \G(0)} (\G[\vare, 1]).$$

It was observed in \cite{RW} that the Loewner energy depends only on its trace (in particular, not on its parametrization) which is a priori not obvious from the definition.
To understand the presence of these symmetries, three identities of the Loewner energy are established in \cite{W2}. 
For the purpose of the present work, let us review the link to the determinants of Laplacians.

Let $g$ be a Riemannian metric on the $2$-sphere $S^2$ and $\G$ a smooth Jordan curve on $S^2$.
We define
$$\mc H(\Gamma, g) : = \log \detz'(-\D_{S^2, g}) -  \log \vol_g(M)- \log \detz (-\D_{D_1, g}) -  \log \detz (-\D_{D_2,g}), $$
where $D_1, D_2 \subset S^2$ are the connected components of $S^2 \backslash \G$, $\D_{D_i, g}$ the  Laplacian on $D_i$ with Dirichlet boundary condition and $\detz'$ (resp. $\detz$) the zeta-regularized determinant of operators with non-trivial (resp. trivial) kernel. 
The zeta-regularization is introduced by Ray and Singer \cite{RaySinger1971}. More details on the functional $\mc H$ can be found e.g. in Section~7 of \cite{W2}. 

\begin{theorem}[\cite{W2} Proposition 7.1, Theorem 7.3] \label{thm_energy_det}
If $g = e^{2\s} g_0$ is conformally equivalent to the spherical metric $g_0$  on $S^2$, then
\begin{enumerate}[(i)]
\item  $\mc H (\cdot, g) = \mc H(\cdot, g_0)$;
\item \label{item_circle_minimize} circles minimize $\mc H(\cdot, g)$
 among all smooth Jordan curves;
\item \label{item_energy_determinant}
we have the identity
\begin{equation} \label{eq_energy_det}
I^L(\G, \G(0)) = 12 \mc H(\G, g) - 12 \mc H(S^1, g).
\end{equation}
\end{enumerate}
\end{theorem}

The right-hand side of \eqref{eq_energy_det} clearly does not depend on the parametrization of $\G$.
~\\

It was pointed out in \cite{Dubedat2009partition} and \cite{LeJan2006det} that 
$-\log \detz (-\D_M)$ can be thought as a renormalization of the total mass of Brownian loops contained in $M$ under the Brownian loop measure $\mu_M^{loop}$ introduced by Lawler and Werner \cite{LW2004loupsoup}. 
Roughly speaking,
   $$``- \log \detz (-\D_M) = \int d \mu_{M}^{loop}".$$
   Hence the identity \eqref{eq_energy_det} can be interpreted as
   $$``I^L(\Gamma) = 12 \mu_{\m C}^{loop} (\d, \d \cap S^1 \neq \emptyset)- 12 \mu_{\m C}^{loop} (\d, \d \cap \G \neq \emptyset). " $$
   However, both terms on the right-hand side diverge.
   In this work, we make sense of the above identity for all finite energy Jordan curves as a renormalization of Werner's measure which is closely related to the Brownian loop measure. 
   The precise statement is in Theorem~\ref{thm_cut_off}.
As an intermediate step, we derive a variation formula for the Loewner energy under a conformal mapping (Theorem~\ref{thm_restr_loop}) which also has independent interest.

Let us make some very loose comments at the end of this introduction.
The definition of the Loewner energy was motivated by the fact that $I(\cdot)$ is the large deviation rate function of the driving function of $\SLE_\k$. 
Therefore, the Loewner energy quantifies the decay of a certain volume of the infinitesimal neighborhood of a given curve measured by the total mass of $\SLE_{0+}$ contained in the neighborhood.
In the same line of thought, the results in the present work also describe the Loewner energy as a measurement of the total mass of Brownian loops contained in the infinitesimal neighborhood. 
\vspace{20 pt}

\noindent{\bf Acknowledgement:}
   I thank Wendelin~Werner for inspiring discussions and his help to improve the manuscript.

\section{Werner's measure on self-avoiding loops}
In this section, we briefly recall the definition and main features of Werner's measure on self-avoiding loops.  

  The \emph{Brownian loop measure} on the complex plane $\m C$ and its sub-domains 
  has been introduced by Lawler and Werner in \cite{LW2004loupsoup} and its definition can be immediately extended to a general Riemannian surface $M$ (with or without boundary) in the following way: 
  
  Let $x \in M$, $t > 0$, consider the sub-probability measure $\m W^t_x$ on the path of the Brownian motion (diffusion of infinitesimal generator the Laplace-Beltrami operator $\D_M$) on $M$ started from $x$ on the time interval $[0,t]$, which is killed if it hits the boundary of $M$. The measures $W^t_{x\to y}$ on paths from $x$ to $y$ are obtained from the disintegration of $\m W^t_x$ according to its endpoint $y$:
  $$\m W^t_x = \int_{M} \m W^t_{x\to y} \dd\vol (y).$$
  Define the \emph{Brownian loop measure} on $M$:
  $$\mu_M^{loop}: = \int_0^{\infty} \frac{\dd t}{t} \int_M \m W_{x \to x}^t \dd \vol(x).$$
  Since the starting points coincide with the endpoints, it is a measure on the set of unrooted loops
   by forgetting the starting point and the time-parametrization (so that we distinguish loops only by their trace). 
   
   The Brownian loop measure satisfies the following two remarkable properties
   \begin{itemize}  
     \item \emph{(Restriction property)} If $M' \subset M$, then 
     $d \mu_{M'}^{loop} (\d)= 1_{\g \in M'}  d \mu_M^{loop} (\d)$.
     \item \emph{(Conformal invariance)} On the surfaces $M_1 = (M, g)$ and $M_2 = (M, e^{2\sigma} g)$ be two conformally equivalent Riemann surface, where $\sigma \in C^{\infty} (M, \m R)$, then
     $$\mu^{loop}_{M_1} = \mu^{loop}_{M_2}.$$
   \end{itemize}
   Notice that the total mass (under the Brownian loop measure) of loops contained in $\m C$ is infinite (in fact, that for all positive $R$, both the mass of loops of diameter greater than $R$ and the mass of loops of diameter smaller than $R$ are both infinite),  which can be viewed as a consequence of its scale-invariance (or from the fact that the integral of $1/t$ diverges both at infinity 
   and at $0$). 
  However,  when $D \subset \m C$ is a proper subset of $\m C$ with non-polar boundary, and $K_1, K_2$ are two disjoint compact subsets of $D$, the total mass (under the Brownian loop measure) of the set of loops 
  that do stay in $D$ and intersect both $K_1$ and $K_2$ is finite (staying in $D$ in some sense removes most large loops, and intersecting both $K_1$ and $K_2$ prevents the loops for being too small). 
 We will denote this finite mass by 
$$\mc B(K_1, K_2; D) := \mu^{loop}_D (\{\d; \d \cap K_1 \neq \emptyset, \d \cap K_2 \neq \emptyset\}).$$

\emph{ Werner's measure} on simple (self-avoiding) loops in the complex plane defined in \cite {Werner2008measure} is simply the image of $\mu_{\m C}^{loop}$ under that map that associates to a (Brownian) loop its outer boundary (i.e., 
 the boundary of the unbounded connected component of its complement). 
 As shown in \cite {Werner2008measure}, this measure turns out to be invariant under the map $z \mapsto 1/z$ 
 (and more generally under any conformal automorphism of the Riemann sphere), which in turn makes it possible to define this measure $\mu^{loop}_{W, M}$ in any Riemann surface $M$, in such a way that the above restriction and conformal invariance properties still hold for this family of measures on self-avoiding loops.  
 In fact, shown in \cite {Werner2008measure} that this is the unique (up to a multiplicative constant) such family of measures on self-avoiding loops satisfying both the restriction property and the conformal invariance properties.
For other characterizations (via a restriction-type formula, or as a measure on SLE loops) of Werner's measure and its properties (it is supported on SLE$_{8/3}$-type loops which have fractal dimension $4/3$), 
see \cite {Werner2008measure}.  
Since we will be discussing conformal restriction properties of the Loewner energy here, 
it is worth stressing here that the proofs in \cite {Werner2008measure} are building on the work of Lawler, Schramm and Werner \cite {LSW2003restriction} on chordal conformal restriction properties. 
   
One feature  that makes Werner's measure convenient to work with on Riemann surfaces is that 
if we consider two disjoint compact sets $K_1, K_2 \subset \m C$, then the total mass of loops that intersect both $K_1$ and $K_2$ is finite (see \cite{NacuWerner2011} Lemma~4):  
$$\mc W(K_1, K_2; \m C) := \mu^{loop}_{W,\m C} (\{\d; \d \cap K_1 \neq \emptyset, \d \cap K_2 \neq \emptyset\})<\infty.$$
This contrasts with the fact that the total mass (for the Brownian loop measure) of loops that intersect both $K_1$ and $K_2$ is infinite, due to the many very large 
Brownian loops that intersect both $K_1$ and $K_2$ (but the outer boundary of these large loops 
tends not to intersect $K_1$ or $K_2$, which explains why $\mc W (K_1, K_2 ; \m C) $ is finite). 
This feature was also one motivation for \cite {Werner2008measure} was for instance instrumental in the proof of the conformal invariance of simple Conformal Loop Ensembles on the Riemann sphere 
by Kemppainen and Werner in \cite{KemWer2016CLE}.

\section{Conformal restriction for simple chord}

We first recall the variation formula of the chordal Loewner energy under conformal restriction, first appeared in \cite{Dubedat2007commutation} and \cite{W1}: 
Let $K$ be a compact hull in $\m H$ at positive distance to $0$. 
The simply connected domain $H_K : = \m{H} \backslash K$ coincides with $\m{H}$ in the neighborhoods of $0$ and $\infty$. 
Let $\G$ be a simple chord contained in $H_K$ connecting $0$ to $\infty$ with finite Loewner energy in $(\m H, 0, \infty)$.

\begin{prop}[\cite{W1} Proposition~4.1]
\label{prop_energy_change}
The energy of $\G$ in $(\m{H},0 ,\infty)$ and in $(H_K, 0 , \infty)$ differ by 
\begin{align*}
 I_{H_K, 0 , \infty}(\G) - I_{\m H, 0, \infty}(\G) & = 3\log \abs{\psi'(0) \psi'(\infty)} + 12 \mc B(\G, K; \m{H}) \\
   & = 3\ln\abs{\psi'(0) \psi'(\infty)} + 12 \mc W(\G, K; \m{H}), 
\end{align*}
where $\psi$ is a conformal map $H_K \rar \m{H}$ fixing $0,\infty$.
\end{prop}

Notice that the derivatives of $\psi$ at boundary points $0$ and $\infty$ are well-defined by Schwarz reflection principle since $H_K$ coincides with $\m H$ in their neighborhood.
The first equality is the analogy of the conformal restriction property of SLE derived in \cite{LSW2003restriction}.
The second equality is due to the fact that $\m H$ is simply connected domain with non-polar boundary and both $K$ and $\G$ are attached to the boundary, so that the Brownian loop hits both $K$ and $\G$ if and only if the outer-boundary hits them.
For readers' convenience, we include the derivation of the first equality below. 

Without loss of generality, we choose the conformal map $ \psi : H_K \to H$ as in Proposition~\ref{prop_energy_change} such that $\psi'(\infty) = 1$. 
Let $K_t$ be the image of $K$ under the flow $g_t$ associated to $\Gamma$, $\tilde g_t$ the mapping-out function of $\psi (\G[0,t])$, and $\psi_t = \tilde g_t \circ \psi \circ g_t^{-1}: \m{H} \backslash K_t \rar \m{H}$ making the diagram commute (see figure~\ref{fig_diagram}).
It suffices to show that for $T < \infty$,
   \begin{equation}\label{eq_energy_change_finite}
   I_{\m H, 0, \infty}(\psi(\G[0,T])) - I_{\m H, 0, \infty}(\G[0,T]) = 3\ln\abs{\psi'(0)} + 12 \mc B (\G[0,T], K; \m{H}) - 3\ln\abs{\psi_T'(0)}
   \end{equation}
which implies Proposition~\ref{prop_energy_change} since the last term $3\ln\abs{\psi_T'(0)} \rar 0$ when $T \rar \infty$ and $ I_{\m H, 0, \infty}(\psi(\G[0,T])) =  I_{H_K, 0 , \infty}(\G[0,T])$.

 \begin{figure}[ht]
 \centering
 \includegraphics[width=0.7\textwidth]{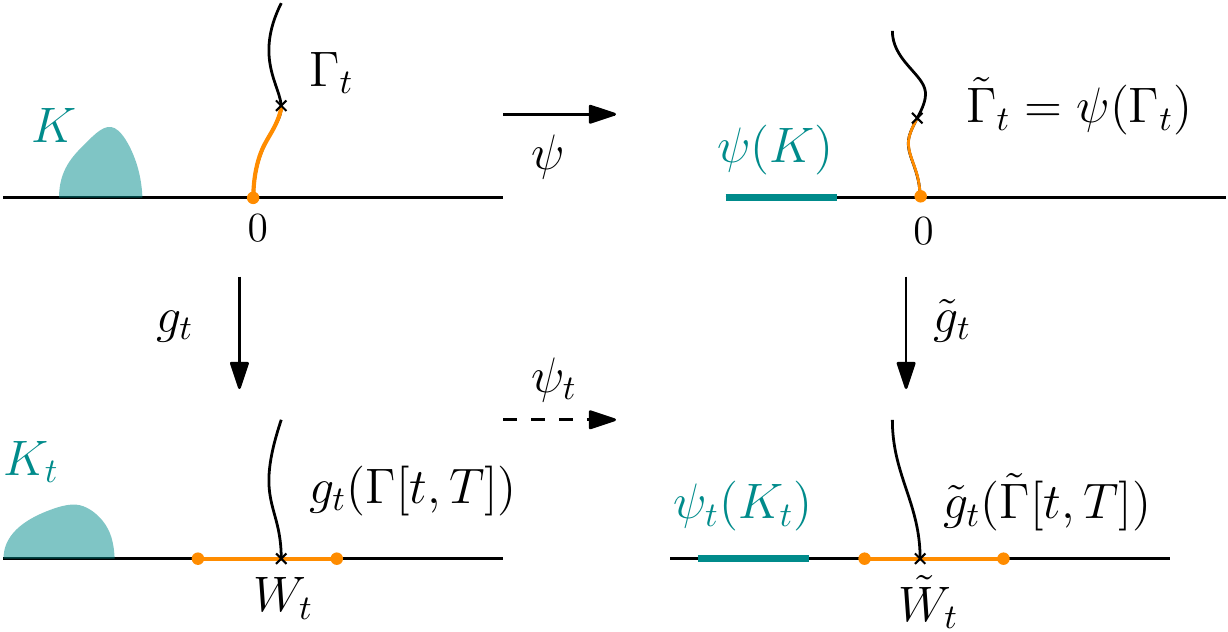}
 \caption{\label{fig_diagram} Maps in the proof of Proposition~\ref{prop_energy_change}, $\tilde W_t = \psi_t(W_t)$.}  
 \end{figure}

\begin{proof}
We write $W_t$ for $W(t)$ to shorten the notation.  We show first
 \begin{equation}
 I_{\m H, 0, \infty}(\psi (\G[0,T])) =  \frac{1}{2} \int_0^T \left[ \partial_t W_t - \frac{3 \psi_{t}''(0)}{\psi_{t}'(0)}\right]^2 \dd t. 
 \end{equation}
 Notice that $(\tilde \G_t : = \psi (\G_t))_{t \ge 0}$ is not capacity-parametrized. 
 We denote $a(t)$ the capacity of $\tilde \G[0,t]$, such that 
 $$\tilde g_t (z) = z + 2a(t)/z + o(1/z), \text{ as } z \to \infty.$$
By a scaling consideration, we have
 $\partial_t a(t) = [\psi_t' (W_t)]^2.$
 The family of conformal maps $\tilde g_t$ satisfies the Loewner differential equation: for $z \in \m H$,
 $$\partial_t \tilde g_t(z) = \partial_a \tilde g_t(z) \partial_t a(t) = \frac{2 [\psi_t' (W_t)]^2}{\tilde g_t (z) - \tilde W_t}.$$
 Now we compute the variation of $\tilde W$.
Since $\psi_t$ is defined by $\tilde g_t \circ \psi \circ g_t^{-1}$, we have
\begin{align} \label{eq_partial_t_psi}
\begin{split}
    \partial_t \psi_t (z) & = \partial_t \tilde g_t (\psi \circ g_t^{-1}(z)) + (\tilde g_t \circ \psi)' (g_t^{-1} (z) )\partial_t(g_t ^{-1}(z)) \\
    & = \frac{2[\psi_t' (W_t)]^2}{\tilde g_t \circ \psi \circ g_t^{-1}(z) - \tilde W_t} + (\tilde g_t \circ \psi)' (g_t^{-1} (z) )\frac{-2 (g_t^{-1})'(z)}{z-W_t}\\
    & =\frac{2[\psi_t' (W_t)]^2}{\psi_t (z) - \tilde W_t} - \frac{2 \psi_t' (z)}{z - W_t}.
\end{split}
\end{align}
Expanding $\psi_t$ in the neighborhood of $W_t$ (this is possible since $\psi_t$ is analytic by Schwarz reflection principle), we obtain
$$\partial_t \psi_t(z) =  - 3 \psi_t ''(W_t) + O(z- W_t).$$
Therefore
\begin{align*}
    \partial_t \tilde W_t & = \partial_t (\psi_t(W_t)) = (\partial_t \psi_t) (W_t) + \psi_t(W_t) \partial_t W_t\\
    & = \left(-3 \frac{\psi_t''(W_t)}{\psi_t'(W_t)} + \partial_t W_t\right) \psi_t'(W_t).
\end{align*}
Notice that since we assumed that $\G$ has finite Loewner energy in $\m H$, it implies that $W$ is in $W^{1,2}$. 
In particular, $W$ is absolutely continuous. It is not hard to see that it implies that $\tilde W$ is also absolutely continuous and the above computation of $\partial_t \tilde W_t$ makes sense.
The Loewner energy of $\psi(\G [0,T])$ is given by
$$\frac{1}{2}\int_0^{a(T)} \abs{\partial_a \tilde W(t(a))}^2 d a = \frac{1}{2}\int_0^{T} \abs{\partial_t \tilde W(t)}^2 (a'(t))^{-1} d t  = \frac{1}{2}\int_0^{T} \left[-3 \frac{\psi_t''(W_t)}{\psi_t'(W_t)} + \partial_t W_t\right]^2 d t $$
as we claimed.

Now we relate the right-hand side of \eqref{eq_energy_change_finite} with the mass of Brownian loop measure attached to both $\G$ and $K$.
Differentiating \eqref{eq_partial_t_psi} in $z$ and taking $z \to W_t$, we obtain
$$(\partial_t \psi_t')(W_t) = \frac{\psi_t''(W_t)^2}{2\psi_t'(W_t)} - \frac{4 \psi_t'''(W_t)}{3}.$$
We have also
$$\partial_t [\ln \psi_t'(W_t)] = \frac{1}{2}\left(\frac{\psi_t''(W_t)}{\psi_t'(W_t)}\right)^2 - \frac{4}{3}\frac{\psi_t'''(W_t)}{\psi_t'(W_t)} + \frac{\psi_t''(W_t)}{\psi_t'(W_t)} \partial_t W_t.$$
Therefore
\begin{align}
\label{eq_diff_integrand}
\begin{split}
     & \frac{1}{2} \left[ \partial_t W_t -3 \frac{\psi_t''(W_t)}{\psi_t'(W_t)} \right]^2 -  \frac{1}{2}(\partial_t W_t)^2 
   = \frac{9}{2}\left(\frac{\psi_t''(W_t)}{\psi_t'(W_t)}\right)^2 -3 \frac{\psi_t''(W_t)}{\psi_t'(W_t)}  \partial_t W_t  \\
    = & -3 \partial_t [\ln \psi_t'(W_t)] - 4 S\psi_t (W_t),
\end{split}
\end{align}
where $$S\psi_t =\frac{ \psi_t'''}{\psi_t'} - \frac{3}{2}\left(\frac{\psi_t''}{\psi_t'}\right)^2$$
is the Schwarzian derivative of $\psi_t$. 
Intergrating \eqref{eq_diff_integrand} over $[0,T]$, we obtain the identity
\eqref{eq_energy_change_finite} by identifying the Schwarzian derivative term using the path decomposition of the Brownian loop measure (see \cite{LSW2003restriction,LW2004loupsoup})
$$-4\int_0^T S\psi_t(W_t) dt = 12 \mc B(\G[0,T], K; \m H).$$
\end{proof}

From Proposition~\ref{prop_energy_change}, we deduce a more general relation of Loewner energy of the same chord in two domains $D$ and $D'$ which coincide in a neighborhood of both marked boundary points, by comparing to the Riemann surface $D \sqcup D'$ identified along the connected component of $D \cap D'$ containing $\G$:

\begin{cor} \label{cor_restr_chordal}
   Let $(D,a,b)$ and $(D',a,b)$ be two simply connected domains in $\m C$ coinciding in a neighborhood of $a$ and $b$, and $\G$ a simple curve in both $(D,a,b)$ and $(D',a,b)$. Then we have 
   \begin{align*}
   I_{D', a, b}(\G) - I_{D, a, b}(\G) =&  3\log \abs{\psi'(a) \psi'(b)} \\ 
   +&  12 \mc W( \G, D\backslash D'; D) - 12 \mc W(\G, D'\backslash D; D'),
   \end{align*}
   where $\psi :  D' \to D$ is a conformal map fixing $a$ and $b$.
\end{cor}

\section{Conformal restriction for simple loop}

We prove in this section the following conformal restriction formula for the loop energy. The loop version has the advantage compared to the chordal case of no longer having the boundary terms.

\begin{thm}\label{thm_restr_loop} If $\eta$ is a Jordan curve with finite energy and $\G = f(\eta)$, where $f: A \to \tilde A$ is conformal on a neighborhood $A$ of $\eta$, then
  $$ I^L(\Gamma) - I^L(\eta) = 12 \mc W ( \eta,  A^c ;\m C) - 12 \mc W ( \Gamma, \tilde A^c ; \m C).  $$
\end{thm}
 The loop terms are finite since we can replace $A^c$ by its boundary. 
Both $\partial A^c$  and $\eta$ are compact and are at positive distance.

\begin{rk} 
The right-hand side of the above identity remains the same if we replace $A$ by a subset $B$ such that $\eta \subset B \subset A$.
In fact, since $A \backslash B$ is at positive distance with $\eta$, we have  $\mc W (\eta, A \backslash B; A) < \infty$. 
 We then decompose the loop measure 
$$\mc W (\eta, B^c; \m C) = \mc W (\eta, A^c; \m C) + \mc W (\eta, A \backslash B; A).$$
The conformal invariance of Werner's measure provides that
$$\mc W (\eta, A \backslash B; A) = \mc W (f(\eta),  f(A \backslash B); f(A)) = \mc W (\G, \tilde A \backslash f(B); \tilde A).$$
Hence
$$I^L(\G) - I^L(\eta) = 12 \mc W (\eta, B^c; \m C) - 12 \mc W (\G, f(B)^c; \m C) $$
given that the formula holds for $A$.
\end{rk}

Now we can prove Theorem~\ref{thm_restr_loop}:
 \begin{proof}
 From the remark, we assume that $A$ is an annulus without loss of generality.
 Since the loop energy is a limit of chordal Loewner energies, the idea is to bring back to the conformal restriction in the chordal framework.
 
 More precisely, let $a, b$ be two points on the curve $\eta$, $\tilde a = f(a)$ and $\tilde b = f(b)$.  Let $D: = \Chat \backslash (ab)_{\eta}$ and $\tilde D := \Chat \backslash (\tilde a \tilde b)_{\G}$.
 We take a ``stick'' $T$ attached to the arc $(ab)_{\eta}$ of the curve $\eta$, such that $ D \backslash K$ is simply connected, where $K = A^c \cup T$ is the union of two ``lollipops''. Define $\tilde T := f(T) \subset \tilde A$ the conformal image of $T$ and similarly $\tilde K = \tilde A^c \cup \tilde T$ (see Figure~\ref{fig_restr_loop}). 
  
 \begin{figure}[ht]
 \centering
 \includegraphics[width=0.8\textwidth]{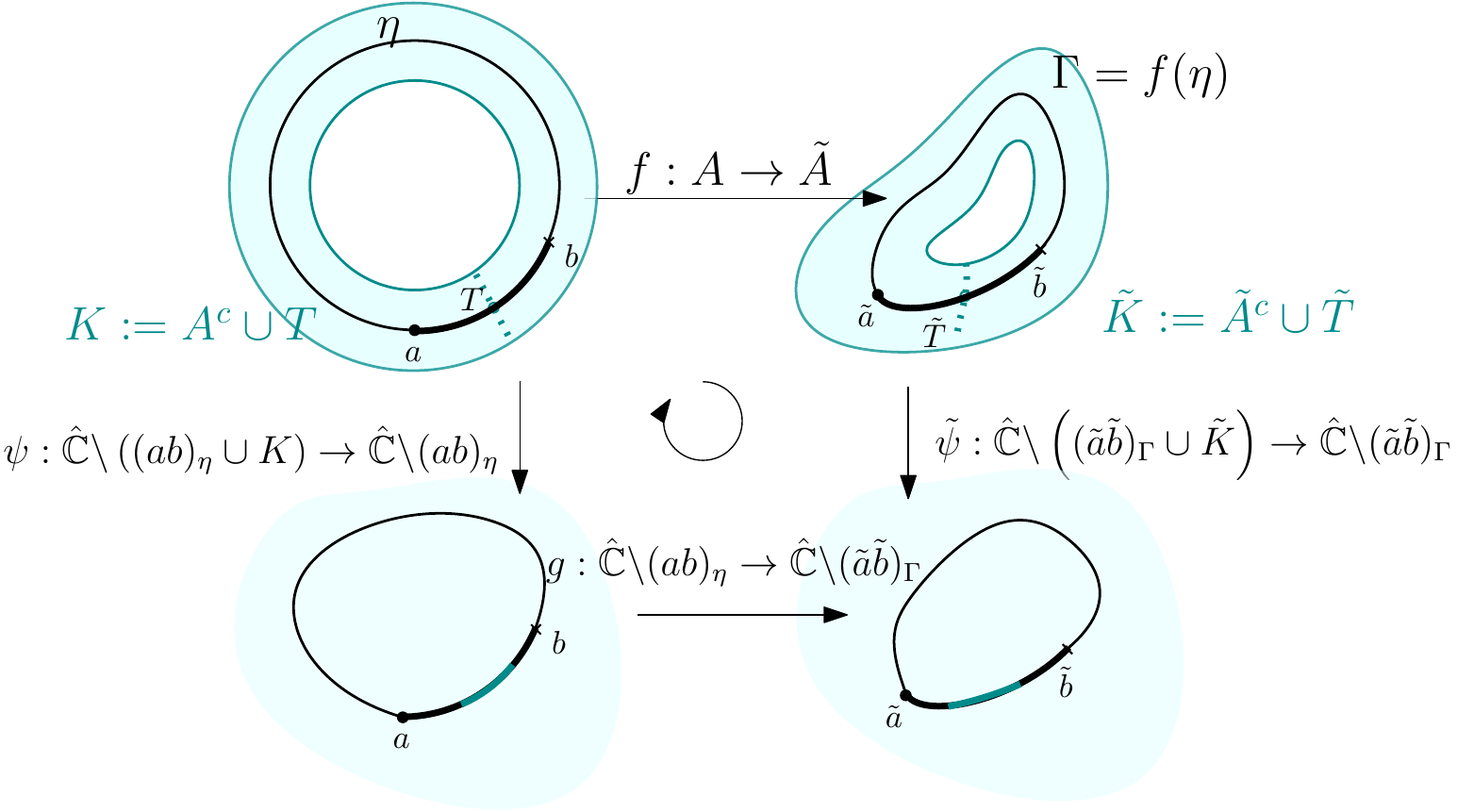}
 \caption{\label{fig_restr_loop} Maps in the proof of Theorem~\ref{thm_restr_loop}.}  
 \end{figure}
  
Now we compare the chordal Loewner energy of $(ba)_{\eta}$ (the complement of $(ab)_{\eta}$ in the curve $\eta$) in $D$ and $(\tilde b \tilde a)_{\G}$ in $\tilde D$.
Notice that $D$ and $D \backslash K$ coincide in a neighborhood of both $a$ and $b$. Let $\psi$ and $\tilde \psi$ be a choice of conformal maps as in Figure~\ref{fig_restr_loop}, and $g$ factorizes the diagram.
Applying Corollary~\ref{cor_restr_chordal} to $(ba)_{\eta}$ in $D$, we have
\begin{equation}\label{eq_var_1}
I_{D \backslash K} ((ba)_{\eta}) - I_{D} ((ba)_{\eta}) = 3 \log \abs{\psi'(a) \psi'(b)} + 12 \mc W ((ba)_{\eta}, K; D),
\end{equation}
and similarly,
\begin{equation}\label{eq_var_2}
I_{\tilde D \backslash \tilde K} ((\tilde b \tilde a)_{\G}) - I_{\tilde D} ((\tilde b \tilde a)_{\G}) = 3 \log \abs{\tilde \psi'(\tilde a) \tilde \psi'(\tilde b)} + 12 \mc W ((\tilde b \tilde a)_{\G}, \tilde K; \tilde D).
\end{equation}

 From the construction, 
 $$I_{D\backslash K} ((ba)_{\eta})  = I_{D} (\psi [(ba)_{\eta}]) = I_{\tilde D} (g \circ \psi ( [(ba)_{\eta}]) = I_{\tilde D} (\tilde \psi [(\tilde b \tilde a)_{\G}]) =    I_{\tilde D \backslash \tilde K} ((\tilde b \tilde a)_{\G}), $$
 where the second equality follows from the conformal invariance of the chordal Loewner energy.
 
We write $H(a,b; D)$ for the Poisson excursion kernel between two boundary points $a,b$ of the domain $D$ (relatively to the local analytic coordinates). Choosing the same analytic coordinates near $a,b$ in the above four pictures, then we have 
 \begin{align*}
 \frac{\psi'(a) \psi' (b)}{\tilde \psi' (\tilde a) \tilde \psi'(\tilde b)} &= \frac{H (a,b; A\backslash (ab)_{\eta})}{H (a,b; D)} \frac{H (\tilde a,\tilde b; \tilde D)}{H (\tilde a,\tilde b; \tilde A \backslash (\tilde a\tilde b)_{\G})} \\
 &= \frac{H (a,b; A\backslash (ab)_{\eta})}{H (\tilde a,\tilde b; \tilde A \backslash (\tilde a\tilde b)_{\G})} \frac{H (\tilde a,\tilde b; \tilde D)}{H (a,b; D)} = \frac{f'(a) f'(b)}{g'(a) g'(b)}
 \end{align*}
 which no longer depends on the stick $T$ chosen.
 
 Also notice that we can decompose the loop measure term as in the remark:
 $$\mc W((ba)_{\eta}, K; D) = \mc W ((ba)_{\eta}, A^c; D) + \mc W ((ba)_{\eta}, T; A \backslash (ba)_{\eta}).$$
 Since the Werner's measure is conformally invariant, we have in particular
 $$\mc W ((ba)_{\eta}, T; A \backslash (ba)_{\eta}) = \mc W ((\tilde b\tilde a)_{\G}, \tilde T; \tilde A \backslash (\tilde b\tilde a)_{\G}).$$
 Taking the difference \eqref{eq_var_1} - \eqref{eq_var_2} combining the above observations, we get
 \begin{align}\label{eq_diff_ab}
 \begin{split}
   & I_{\tilde D} ((\tilde b \tilde a)_{\G})- I_{D} ((ba)_{\eta}) \\
 = & 3 \log \abs{\frac{f'(a)f'(b)}{g'(a) g'(b)}} + 12 \mc W ((ba)_{\eta}, A^c; D) - 12 \mc W ((\tilde b \tilde a)_{\G}, \tilde A^c; \tilde D). 
 \end{split}
 \end{align}
 We conclude the proof by taking $b \to a$ on $\eta$, using the definition of loop energy
 $$I_{D} ((ba)_{\eta}) \xrightarrow[]{b \to a} I^L(\eta, a) = I^L(\eta) $$
 and the fact that
 $$\mc W ((ba)_{\eta}, A^c; D) \xrightarrow[]{b \to a} \mc W(\eta, A^c; \m C).$$
 The log-derivative terms goes $0$ thanks to the following lemma and concludes the proof of the theorem. 
\end{proof}

\begin{lem}\label{lem_anal}
   With the same notations as in the proof of Theorem~\ref{thm_restr_loop} (see Figure~\ref{fig_restr_loop}), 
   $$\lim_{b \to a} \abs{\frac{f'(a)f'(b)}{g'(a) g'(b)}} = 1.$$
\end{lem}
\begin{proof}
Without loss of generality, we may assume that $A = \m D$ the unit disk, $a = 0$, $\tilde a = f(a) = 0$.
 \begin{figure}[ht]
 \centering
 \includegraphics[width=0.9\textwidth]{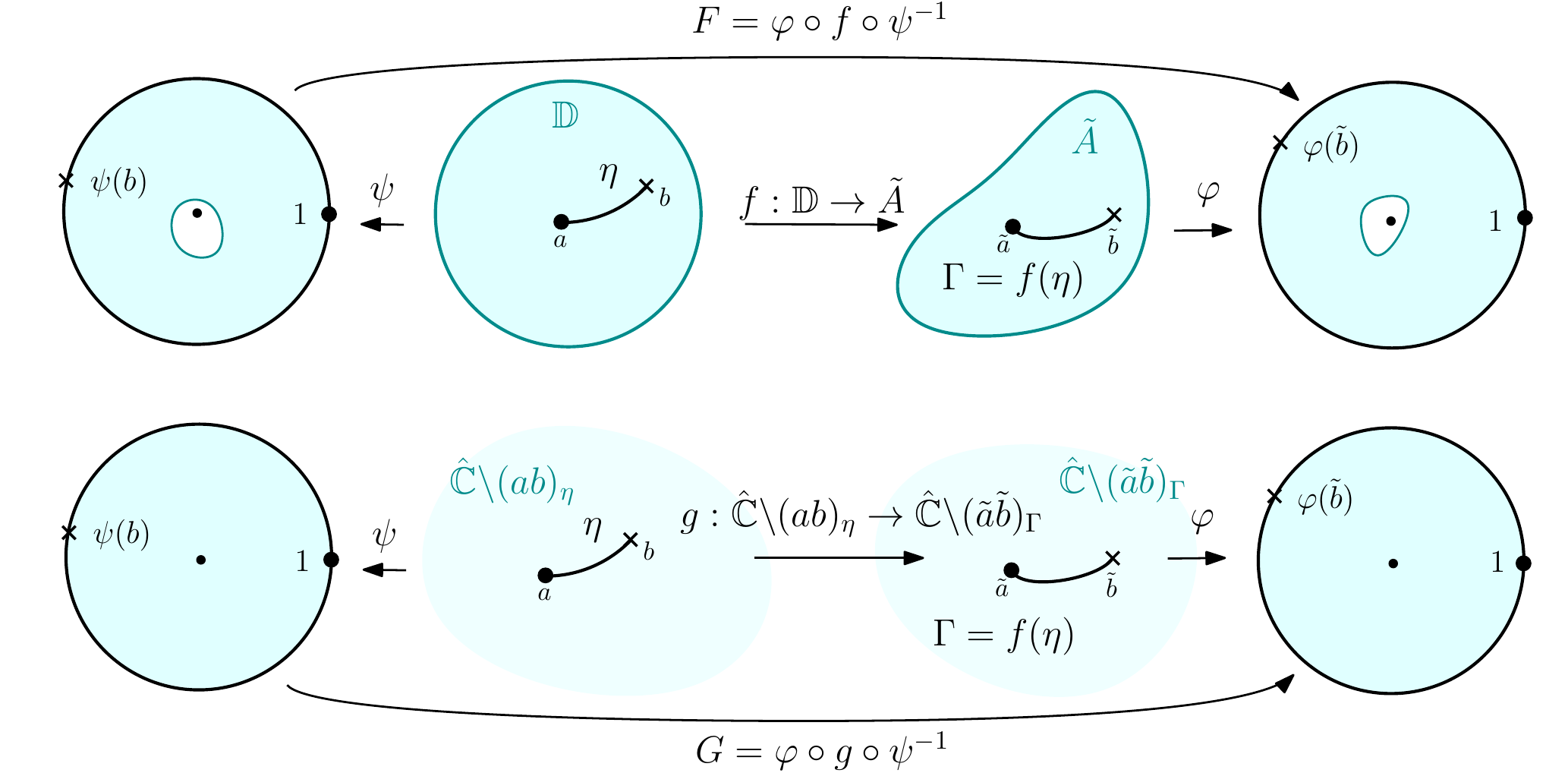}
 \caption{\label{fig_lemma} Maps in the proof of Lemma~\ref{lem_anal}.}  
 \end{figure}
 Let $\psi$ be the conformal map $\Chat \backslash (ab)_{\eta} \to \m D$, such that $\psi (\infty) = 0$, $\psi (0) = 1$. 
 Similarly let $\varphi$ be the conformal map $\Chat \backslash (\tilde a\tilde b)_{\G} \to \m D$, such that $\varphi (\infty) = 0$, $\varphi (0) = 1$.
 Define $F =  \varphi \circ f \circ \psi^{-1}$ and $G=  \varphi \circ f \circ \psi^{-1}$ between the blue-shaded area in Figure~\ref{fig_lemma}.
 
 It is not hard to see that 
 the diameter of $\psi (\m D^c)$ and $\varphi (\m D^c)$ shrinks to $0$ as  $b \to a$. 
 Therefore
 $$\frac{F'(1) F'(\psi(b))}{G'(1) G'(\psi(b))} \xrightarrow[]{b\to a} 1.$$
 On the other hand,
 $$\frac{F'(1) F'(\psi(b))}{G'(1) G'(\psi(b))} = \frac{H (a,b; A\backslash (ab)_{\eta})}{H (\tilde a,\tilde b; \tilde A \backslash (\tilde a\tilde b)_{\G})} \frac{H (\tilde a,\tilde b; \tilde D)}{H (a,b; D)} = \frac{f'(a) f'(b)}{g'(a) g'(b)}$$
 which concludes the proof.
\end{proof}

By taking $\eta = S^1$,
we deduce immediately the interpretation of the Loewner energy of an analytic Jordan curve: 

\begin{cor} \label{cor_energy_anal} If $\G = f(S^1)$ is an analytic curve, then 
$$I^L(\G) = 12 \mc W ( S^1,  A^c ;\m C) - 12 \mc W ( \Gamma, \tilde A^c ; \m C),  $$
where $f : A \to \tilde A$ maps conformally a neighborhood $A$ of $S^1$ to a neighborhood $\tilde A$ of $\G$.
\end{cor}
\begin{proof}
  This follows immediately from Theorem~\ref{thm_restr_loop} and that $I^L (S^1) = 0$.
\end{proof}

\section{Loewner energy as a renormalization of Werner's measure}
 
 Let $\G$ be a Jordan curve in $\m C$, $D$ the bounded connected component of $\m C \backslash \G$ and $f$ a conformal map from the unit disk $\m D$ to $D$. For $1 > \vare > 0$, let $S^{(1-\vare)}$ denote the circle of radius $1-\vare$, centered at $0$, and $\G^{(1-\vare)} : = f(S^{(1-\vare)})$ the equi-potential. 

\begin{thm} \label{thm_cut_off} We have
     $$I^L(\G) = \lim_{\vare \to 0} 12 \mc W(S^1, S^{(1-\vare)}; \m C) - 12 \mc W(\G, \G^{(1-\vare)}; \m C).$$
   \end{thm}
   
   \begin{proof}
   For each $\vare$, we apply Corollary~\ref{cor_energy_anal} to the analytic curve $\G^{(1-\vare)}$ with $A := D$ which gives
   $$I^L(\G^{(1-\vare)}) = 12 \mc W(S^1, S^{(1-\vare)}; \m C) - 12 \mc W(\G, \G^{(1-\vare)}; \m C).$$
   Now it suffices to see that $I^L(\G^{(1-\vare)})$ converges to $I^L(\G)$.
   
   If $I^L(\G) < \infty$, from the geometric description of Loewner energy (see \cite{W2} Section 8) $\G$ is a quasi-circle of the Weil-Petersson class, which is equivalent to
   $$\int_{\m D} \abs{\frac{f''(z)}{f'(z)}}^2 d z^2 < \infty.$$

Let $f_{\vare} (z) : = f((1-\vare) z)$ denote the uniformizing conformal map from $\m D$ to the bounded connected component $\m C \backslash \G^{(1-\vare)}$. We have for $\vare < \vare_0<1/2$,
   \begin{align*}
  & \int_{\m D} \abs{\frac{f''(z)}{f'(z)}- \frac{f_{\vare}''(z)}{f_{\vare}'(z)}}^2 d z^2 \\
   = & \int_{\abs{z} < 1- \vare_0} \abs{\frac{f''(z)}{f'(z)}- \frac{f_{\vare}''(z)} {f_{\vare}'(z)}}^2 d z^2 + \int_{1-\vare_0 \le \abs{z} < 1} \abs{\frac{f''(z)}{f'(z)}- \frac{f_{\vare}''(z)} {f_{\vare}'(z)}}^2 d z^2 \\
   \le & \int_{\abs{z} < 1- \vare_0} \abs{\frac{f''(z)}{f'(z)}- \frac{f_{\vare}''(z)} {f_{\vare}'(z)}}^2 d z^2 + 4 \int_{1-2 \vare_0 \le \abs{z} < 1} \abs{\frac{f''(z)}{f'(z)}}^2 d z^2\\
   & \xrightarrow[]{\vare \to 0} 4 \int_{1-2 \vare_0 \le \abs{z} < 1} \abs{\frac{f''(z)}{f'(z)}}^2 d z^2,
   \end{align*}
   the convergence is due to the fact that $f_{\vare}''/f_{\vare}'$ converges uniformly on compacts to $f''/f'$.
   As $\vare_0 \to  0$, the above integral converges to $0$, and we conclude that
   $$\lim_{\vare \to 0} \int_{\m D} \abs{\frac{f''(z)}{f'(z)}- \frac{f_{\vare}''(z)}{f_{\vare}'(z)}}^2 d z^2 = 0.$$
   It yields that $\G^{(1-\vare)}$ converges in the Weil-Petersson metric to $\G$ (see \cite{TT2006WP} Corollary~A.4 or \cite{W2} Lemma~H) and therefore $I^L(\G^{(1-\vare)})$ converges as well to $I^L(\G)$.
~\\

If $I^L(\G) = \infty$, from the lower-semicontinuity of the Loewner loop energy (\cite{RW}, Lemma 2.9) and the fact that $\G^{(1-\vare)}$ converges uniformly (parametrized by $S^1$ via $f_{\vare}$) to $\G$, we have
$$\liminf_{\vare \to 0} I^L(\G^{(1-\vare)}) \ge I^L(\G) = \infty.$$
Hence $I^L (\G^{(1-\vare)})$ converges to $\infty$ as $\vare \to 0$.
   \end{proof}

\end{document}